\documentclass[10pt, a4paper]{article}
\usepackage[english]{babel}
\usepackage{amsmath}
\usepackage{amsthm}
\usepackage{amssymb}
\usepackage{bbm}
\usepackage{mathrsfs}
\usepackage{stmaryrd}
\usepackage{comandi}
\usepackage[all]{xy}
\usepackage{bussproofs}

\mathcode`\:="603A     % : = \colon
\mathchardef\colon="303A  % :=

\mathcode`\<="4268     % < = \langle
\mathcode`\>="5269     % > = \rangle
\mathchardef\gt="313E  % arithmetic
\mathchardef\lt="313C  % strict order

\theoremstyle{definition}
\newtheorem{deff}{Definition}[section]
\newtheorem{prop}[deff]{Proposition}
\newtheorem{exa}[deff]{Example}
\newtheorem{lem}[deff]{Lemma}

\newtheorem{Remark}[deff]{Remark}

\begin{document}
\title{A co-free construction for elementary doctrines}
\author{Fabio Pasquali}
\maketitle

\begin{abstract}
We provide a co-free construction which adds elementary structure to a primary doctrine. We show that the construction preserves comprehensions and all the logical operations which are in the starting doctrine, in the sense that it maps a first order many-sorted theory into a the same theory formulated with equality. As a corollary it forces an implicational doctrine to have an extentional entailment.
\end{abstract}

\section*{Introduction}

This paper deals with the notion of internal equality in doctrines. Doctrines were introduced by Lawvere (\cite{LawDiag}, \cite{LawAdj}  and \cite{LawEq}) and we pospone in section \ref{sec1} their formal definition. For the purpose of this introduction it is enough to think of doctrines as those presheaves such that, given a theory $\mathcal{T}$ over a many-sorted relational language $\mathcal{L}$, one looks at objects and morphism of the domain category as types and terms of $\mathcal{L}$ respectively, while a well formed formula in $\mathcal{T}$ of type $A$ is an element in the fiber over $A$. Lawvere made extensive use of the language of adjoints and Jacobs \cite{Jacobs} described equality between terms of a given type as a formula in the fiber over the product of that type with itself, satisfying the following rule of inference
\[
\AxiomC{$\Gamma, x:X \mid \phi \vdash \psi[x/y]$}
\doubleLine
\UnaryInfC{$\Gamma, x:X, y:X \mid \phi \wedge x=_Xy \vdash \psi$}
\DisplayProof
\]
where the double line indicates that one of the two sequents holds exactly when the other holds. A doctrine is a first order theory with equality if it possesses a formula $=_X$, for every sort $X$, which satisfies the previous rule.\\ A way to introduce higher order quantification is to consider a new type $\Omega$ in the underlying signature and thinks of terms of type $\Omega$ as propositions. From the categorical viewpoint this generates a correspondence between terms of type $\Omega$ and formulas, and therefore it makes sense to investigate how the notion of internal equality $=_{\Omega}$ is related to logical equivalence. A link is in the following rule, taken from \cite{BoiJoy} and \cite{Jacobs}
\[
\AxiomC{$\Gamma \mid \xi \wedge \phi \vdash \psi$}
\AxiomC{$\Gamma \mid \xi \wedge \psi \vdash \phi$}
\doubleLine
\BinaryInfC{$\Gamma \mid \xi \vdash \phi =_{\Omega} \psi$}
\DisplayProof
\]
where it is implicit that if $\phi$ and $\psi$ are formulas over the context $\Gamma$, then $\phi=_{\Omega} \psi$ is still a formula over $\Gamma$. We say that a doctrine is a higher order many-sorted theory with extentional entailment if there is an object $\Omega$ in the base category and a formula $=_\Omega$ in the fiber over $\Omega \times \Omega$ which satisfies both the previous rules.\\ 
In the present paper we provide a co-free construction that, starting from any doctrine $P$, produces a new doctrine $P_\D$ with equality. That is to say that for every object $X$ in the domain category of $P_\D$ there exists a well formed formula in $P_\D(X\times X)$ which satisfies the first one of the previous rules. We show also that if the starting doctrine $P$ is an higher order implicational theory, the resulting doctrine $P_\D$ will have an internal equality over $\Omega$ satisfying both the previous rules; in other words: $=_{\Omega}$ and logical equivalence comes to coincides.\\
In section \ref{sec1} we give the definitions of doctrines and some relevant examples. In section \ref{sec2} we introduce the construction of Maietti and Rosolini of the category of quotients and the doctrine of descent data which is the base of the co-free construction we are going to provide in \ref{sec3}. In the last section we show which properties are preserved by the construction and some applications.
%%%%%%%%%%%%%%
%%%%%%%%%%%%%%
%%%%%%%%%%%%%%
%%%%%%%%%%%%%%
\section{Doctrines}\label{sec1}
We recall those structures which we will be concerned with in the paper, see \cite{RM} and \cite{Tripinret}.
\begin{deff}
A \textbf{primary doctrine} is a functor $P: \C ^{op} \arr \textbf{ISL}$,
where $\textbf{ISL}$ is the subcategory of \textbf{Posets} consisting of inf-semilattices and homomorphisms and $\C$ is a category with binary products.
\end{deff}
For the rest of the paper we will write $f^*$ instead of $P(f)$, to indicate the action of the functor $P$ on a morphism $f$ of $\C$. We shall refer to $f^*$ as the reindexing functor along $f$. Left and right adjoints to reindexing functor $f^*$ will be $\exists_f$ and $\forall_f$ respectively. We say that a doctrine has finite joins if every fiber has finite joins. Analogously we say that a doctrine is implicational if every fiber has relative pseudo complements which commute with reindexing. For every pair of element $x$ and $y$ we will denote their meet by $x\wedge y$, by $x\lor y$ their join and by $x\imply y$ their relative pseudo complements. Top and bottom elements will be $\top$ and $\bot$ respectively. Joins are said to be distributive if for every $x$, $y$ and $z$ it holds that $x \wedge (y \lor z) = (x \wedge y) \lor (x \wedge z)$.
\begin{deff}\label{ele}
A primary doctrine $P$ is said to be \textbf{elementary} if for every $A$ in $\C$ there exists an object $\delta_A$ in $P(A \times A)$ such that for every $X$ in $\C$
 \begin{itemize}
\item[i)] the assignment $\pi_1^* (\alpha) \wedge \delta_A$ determines a left adjoint to $\Delta_A^*$
\item[ii)] the assignment $\langle \pi_1,\pi_2\rangle^*(\alpha)\wedge \langle\pi_2,\pi_3\rangle^*(\delta_A)$ determines a left adjoint to $(id_X \times \Delta_A)^*$ 
\end{itemize}
\end{deff} 
Primary doctrines are the objects of the 2-category \textbf{PD} in which\\
\\
the \textbf{1-arrows} are pairs $(F,f): P \arr R$
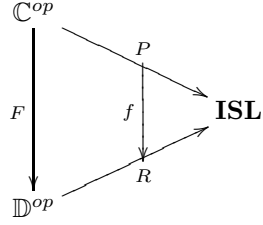
$$
\xymatrix{
\C^{op}\ar[rrd]^-{P}="a"\ar[dd]_{F}&&\\
&&\textbf{ISL}\\
\mathbb{D}^{op}\ar[rru]_-{R}="b"&&
\ar "a";"b"_{f}}
$$
where the functor $F$ preserves products and $f$ is a natural transformation from the functor $P:\C^{op}\arr \textbf{ISL}$ to the functor $R\circ F:\D^{op}\arr \textbf{ISL}$\\
\\
the \textbf{2-arrows} are those natural transformations $\nu$ 
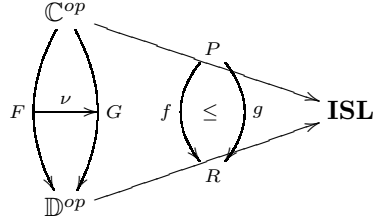
$$
\xymatrix{
\C^{op}\ar[rrrd]^-{P}="a"\ar@/_1pc/ [dd]_{F}="c"\ar@/^1pc/ [dd]^{G}="d"&&&\\
&&&\textbf{ISL}\\
\mathbb{D}^{op}\ar[rrru]_-{R}="b"&&&
\ar@/_1pc/"a";"b"_{f}\ar@/^1pc/"a";"b"^{g}_{\le \ \ }\ar"c";"d"^{\nu}}
$$
such that, for every object $A$ in $\C$ and every $\alpha$ in $P(A)$, it holds that $\nu_A^* (f_A (\alpha)) \le g_A (\alpha)$.\\
\\
We call \textbf{ED} the 2-subcategory of \textbf{PD}, in which the object are elementary doctrines and the 1-arrows are those 1-arrows in \textbf{PD} such that $$f_{A\times A}(\delta_A) = <F\pi_1,F\pi_2>^*\delta_{FA}$$ for every 1-arrows $(F,f)$ and for every object $A$ in $\C$.
\begin{deff}\label{uni} A primary doctrine is called \textbf{universal} if for every projection arrows $\pi$ in $\C$ the functor $\pi^*$ has a right adjoint $\forall _{\pi}$ satisfying Beck-Chevalley condition: given a pullback diagram of the kind
$$
\xymatrix{
X\times Y' \ar[d]_-{id \times f}\ar[r]^-{\pi'}& Y'\ar[d]^-{f}\\
X\times Y \ar[r]_-{\pi}& Y
}
$$
it holds that $\forall_{\pi'} \circ (id\times f)^* = f^* \circ \forall_{\pi}$
\end{deff}
A primary docrine is \textbf{existential} if the reindexing functors along a projection have a left adjoint satisfying Beck-Chevalley and Frobenius reciprocity: $\exists_\pi(\alpha\wedge \pi^*\beta)=\exists_\pi(\alpha)\wedge \beta$, for $\alpha$ in $P(X\times Y)$ and $\beta$ in $P(Y)$.
\begin{Remark}\label{GQ} Recall from \cite{LawAdj,Tripinret} that in an elementary existential doctrine $P$ for every morphism $f:A\arr B$ in the base category there exists a functor $\exists_f : P(B)\arr P(A)$ such that $\exists_f \dashv f^{*}$. Indeed if $\pi_A$ and $\pi_B$ are the projections from $A\times B$ to $A$ and $B$ respectively, for $\alpha$ in $P(A)$
$$\exists_f (\alpha) := \exists_{\pi_B}((id_B \times f)^*\delta_B\wedge \pi_A^*\alpha)$$
Such a generalized quantification satisfies Frobenius Reciprocity. For $\beta$ in $P(B)$, we have that
$(id_B\times f)^*\delta_B \wedge f^*\beta = (id_B\times f)^*\delta_B \wedge \pi_B^*\beta$. Therefore
\begin{displaymath}
\begin{array}{c}
\exists_{\pi_B}((id_B \times f)^*\delta_B\wedge \pi_A^*\alpha \wedge f^*\beta) =\\
\exists_{\pi_B}((id_B \times f)^*\delta_B\wedge \pi_A^*\alpha \wedge \pi_B^*\beta) =\\
\exists_{\pi_B}((id_B \times f)^*\delta_B\wedge \pi_A^*\alpha)\wedge \beta
\end{array}
\end{displaymath}
And for a pullback square such as that in \ref{uni} the Beck-Chevalley condition holds: $\exists_{(id_X \times f)}\ \pi^* = \pi'^*\ \exists_f$.
\end{Remark}
%\begin{deff} 
%Given a primary doctrine $P:\C^{op}\arr \textbf{ISL}$, we say that $P$ has \textbf{comprehensions} if for every object $A$ in $\C$ and every element $\alpha$ in $P(A)$ there exists an arrow $\lfloor \alpha \rfloor : X \arr A$ in $\C$ such that\\
%\\
%$\bullet$ $\top_X \le \lfloor \alpha \rfloor ^* (\alpha)$\\
%\\
%$\bullet$ for every arrow $f:Y\arr A$ in $\C$ such that $\top_Y \le f^*(\alpha)$, there exists a unique $u:Y\arr X$ making the triangle commutes
%$$
%\xymatrix{
%X\ar[rr]^{\lfloor \alpha \rfloor}&&A\\
%Y\ar[rru]_{f}\ar@{.>}[u]^{u}&&
%}
%$$
%\end{deff}
%A comprehension $\lfloor \alpha \rfloor : X \arr A$ of $\alpha$ is said to be \textbf{full} if $\alpha \le \beta$ whenever $\top_X \iso \lfloor \alpha \rfloor ^* (\beta)$.
%We say that a doctrine $P$ has full comprehensions if every comprehension arrow is full.
\begin{deff}
A primary doctrine is said to have a \textbf{weak power objects} if for every $A$ in $\C$ there exists an object $\pi A$ in $\C$ and an element $\in_A$ in $P(A\times \pi A)$ such that, for every object $B$ in $\C$ and element $\phi$ in $P(A \times B)$ there exists a morphism $\{\phi\}: B\arr \pi A$ such that $\phi=(id_A \times \{\phi\})^*\in_A$.
\end{deff}
\begin{Remark}\label{p1}
In the case the base category $\C$ has a terminal object $1$: the first item in the definition \ref{ele} is redundant, since it becomes a particular instance of the second; when the doctrine has weak power objects, for every object $A$ in $\C$ each element $\phi$ in $P(A)$ determines a term of type $\pi1$ via the following  isomorphism:
$$
\xymatrix{
1\times \pi1\ar[r]^{i}&\pi1\\
1\times A\ar[u]^{id_1\times \{j^*\phi\}} \ar[r]_{j}& A\ar[u]_{\{\phi\}}
}
$$
we will denote with  $\epsilon_1$ the element $(i^{-1})^*\in_1$; in the case $\C$ has all comprehensions, defined to be those morphisms $\lfloor \phi \rfloor : X \arr A$ which are terminal with respect to the property that $\top_X \le \lfloor \phi \rfloor^*(\phi)$, for every $\phi$ in $A$ (see \cite{RM}), then $\lfloor \phi \rfloor$ is weakly classified by $\{\phi\}$, where the true arrow is $\lfloor \epsilon_1\rfloor : 1 \arr \pi1$.
\end{Remark}
There are several examples of doctrines, we list a few.
%\begin{exa}\label{syn} (\textbf{Syntactic}) This is the indexed order $LT: V^{op} \arr \textbf{ISL}$ built on the Lindenbaum-Tarski algebras of well-formed formulas of a first order theory (the definition is taken from \cite{R}).\\
%\\
%Given a theory $\mathcal{T}$ in a first order language $\mathcal{L}$, $V$ is the category of lists of distinct variables $\vec{x}$ as objects and substitutions $[\vec{t}/\vec{y}] : \vec{x}\arr \vec{y}$ as morphisms.\\
%An element of $LT(\vec{x})$ is a provable equivalence class of well-formed formula with no more free variables than ones which compare in $\vec{x}$. Order $[\phi]\le[\psi]$ is given by logical entailment    $\phi \vdash_{\mathcal{T}} \psi$ on representatives.
%\end{exa}
\begin{exa}\label{syn} (\textbf{Syntactic}) Given a theory $\mathcal{T}$ in a first order language $\mathcal{L}$, the base category $\mathbb{V}$ has lists of distinct variables $\vec{x}=(x_1, x_2, . . . x_n)$ as objects and lists of substitutions $[\vec{t}/\vec{y}]:\vec{x}\arr \vec{y}$ as morphisms. Composition is given by simultaneous substitution. For an object $\vec{x}$ in $\mathbb{V}$, the fiber over $\vec{x}$ consists of equivalence classes of well-formed formulae of $\mathcal{L}$ with no more free variables than $x_1, x_2, . . . x_n$, with respect to reciprocal entailment of $\mathcal{T}$, see \cite{RM}.
\end{exa}
\begin{exa}\label{sub} (\textbf{Subobjects}) Suppose $\C$ a small category with binary products and pullbacks. Consider the functor that assigns for every object $A$ in $\C$ the collection $\textbf{Sub}(A)$ of subojects with codomain $A$, ordered by factorization. The top element is (the equivalence class of) the identity arrow. A representative of $\alpha \wedge \beta$ is any pullback of $\alpha$ along $\beta$. Given a morphism $f$ in $\C$, $f^*\alpha$ is the class of any pullback of $\alpha$ along $f$. If $\C$ is regular, the doctrine has left adjoints of all reindexing functors. It is elementary with $\delta_A=\Delta_A : A \arr A \times A$. $\textbf{Sub}:\C^{op}\arr \textbf{ISL}$ has full comprehensions. An element $\alpha : X\arr A$ in $\textbf{Sub} (A)$  has itself as its own comprehension. Consider the following diagrams
$$
\xymatrix{
X\ar[d]_-{\top_X}\ar[r]^-{\top_X}&X\ar[d]^-{\alpha}\\
X\ar[r]_-{\alpha}&A
}
\quad \qquad
\xymatrix{
Y\ar[rd]_-{\top_Y}\ar[r]^-{k}&P\ar[r]^-{p}\ar[d]^-{f^*\alpha}&X\ar[d]^-{\alpha}\\
&Y\ar[r]_-{f}&A
}
\quad \qquad
\xymatrix{
X\ar[rd]_-{\top_X}\ar[r]^-{h}&Q\ar[r]^-{q}\ar[d]^-{\alpha^*\beta}&X\ar[d]^-{\beta}\\
&X\ar[r]_-{\alpha}&A
}
$$
where $P$ is a pullback of $\alpha$ along $f$ and $Q$ the pulback of $\beta$ along $\alpha$. The left one is a pullback and says that $\top_X \le \alpha^* \alpha$. The second proves that if $\top_Y \le\ f^*\alpha$ ($\le$ is $k$), then $f$ factorizes through $\alpha$. Third pullback shows that if $\top_X \iso\ \alpha^*\alpha\ \le\ \alpha^*\beta$ ($\le$ is $h$), then $\alpha\ \le\ \beta$ ($\le$ is $q \circ h$).
%If the base category $\C$ is an elementary topos, then \textbf{Mono} has power objects: $\pi A$ is $\Omega ^{A}$. $\in _A$ is the pullback of $\text{true} :1\arr \Omega$ along $ev_A : A \times \Omega^{A} \arr \Omega$. For every mono $y: Q\arr A\times B$ define $\{y\}:B\arr \Omega^{A}$ to be the exponential transpose of the classifying morphism $\chi_{_Q} : A\times B\arr \Omega$. The converse is true in a weaker form. Suppose the doctrine \textbf{Mono} to have power objects. Define $\Omega$ to be $\pi1$. The true arrow is $\in_1: I\arr\pi1$. Every mono $\phi:X\arr A$ is classified, not in a unique way, by $\{\phi\}$, since reindexing is given by pullbacks. For each $A$ in $\C$, $\Omega^{A}$ is $\pi A$ and $ev_A$ is $\{\in_A\}$. For every morphism $f : A\times B\arr \Omega$, the transpose $\overline{f}$ is $\{f^*\in_1\}$. Then $\{\in_A\} \circ id_A\times\{f^*\in_1\}$ and $f$ need not to be the same morphism, but both classify $f^*\in_1$.
%As proved in \cite{MM} pag.59, given $f^*: \C/A \arr \C/A'$, if $\C/A$ has relative pseudo complements, then $f^*$ has a right adjoint.
\\A particular case is when $\C$ is a small, full subcategory of \textbf{Set} closed under binary products and subsets and the functor $\textbf{Sub}:\C^{op}\arr \textbf{ISL}$ coincides with the powerset functor
\end{exa}
\begin{exa}\label{Tripos}(\textbf{Triposes}) We refer to the definition given by Pitts in \cite{Tripinret}.\\
 Given a category $\C$ with binary products, a tripos is a primary doctrine $P: \C ^{op} \arr \textbf{ISL}$ such that:
(i) for every object $A$ in $\C$, $P(A)$ is a Heyting Algebra (ii) for every arrow $f$ in $\C$, $f^*$ is an homomorphism of Heyting algebras (iii) for every projection arrow $\pi$ in $\C$ the functor $\pi^*$ has left and right adjoints satisfying the Beck-Chevalley conditions (iv) $P$ has weak power objects (v) for every object $A$ in $\C$ there exists an element $\delta_A$ in $P(A \times A)$ such that, for all $\alpha$ in $P(A\times A)$, $\top_A \le \Delta_A^*(\alpha)$ if and only if $\delta_A \le \alpha$.\\
 \\
 All triposes are universal doctrine with weak power objects. They are elementary, since 
 the assignment $\exists _{\Delta _X} (\alpha) := \pi_1 ^* (\alpha) \wedge \delta _X$ provides a left adjoint to $\Delta _X ^*$, in fact
\[ \AxiomC{$\alpha \leq \Delta _X ^* (\beta)$}
\UnaryInfC{$\top _ X \leq \alpha \imply \Delta _X ^* (\beta)$}
\UnaryInfC{$\top _ X \leq \Delta_X ^* (\pi _1 ^* (\alpha) \imply \beta)$}
\UnaryInfC{$\delta _X \leq \pi _1 ^* (\alpha) \imply \beta$}
\UnaryInfC{$\delta _X \wedge \pi _1 ^* (\alpha) \leq \beta$}
\DisplayProof
\qquad
\AxiomC{$\exists_{\Delta _X}(\alpha) \leq \beta$}
\UnaryInfC{$\pi_1 ^*(\alpha) \wedge \delta_X \leq \beta$}
\UnaryInfC{$\alpha \wedge \Delta_X ^*(\delta_X) \leq \Delta_X ^*(\beta)$}
\UnaryInfC{$\alpha \wedge \top _X \leq \Delta_X ^*(\beta)$}
\UnaryInfC{$\alpha \leq \Delta_X ^*(\beta)$}
\DisplayProof
\]
and 
the assignment $\exists _e (\alpha) :=  \langle \pi _1, \pi _2\rangle ^* (\alpha) \wedge \langle \pi _2, \pi _3\rangle ^*(\delta _A)$ determines a left adjoint to the reindexing of $e:=id_X \times \Delta_A : X\times A \arr X\times A\times A$
\[ \AxiomC{$\exists _e (\alpha) \leq \beta$}
\UnaryInfC{$e^*\exists _e  (\alpha) \leq e^*(\beta)$}
\UnaryInfC{$\alpha \wedge \langle\pi _2, \pi _2\rangle ^*(\delta _A) \leq e^*(\beta)$}
\UnaryInfC{$\alpha \wedge \pi_2 ^* \Delta_A ^* \exists _{\Delta _A} (\top _A) \leq e^*(\beta)$}
\UnaryInfC{$\alpha \wedge \pi_2 ^* (\top _A) \leq e^*(\beta)$}
\UnaryInfC{$\alpha \leq e^*(\beta)$}
\DisplayProof
\qquad
\AxiomC{$\alpha \leq e^*( \beta)$}
\UnaryInfC{$\top _{X\times A} \leq \alpha \imply e^*(\beta)$}
\UnaryInfC{$\pi_2 ^*\top _A \leq e^* \langle\pi_1,\pi_2\rangle^*(\alpha) \imply e^*(\beta)$}
\UnaryInfC{$\top _A \leq \forall_{\pi_2}e^* (\langle\pi_1,\pi_2\rangle^*(\alpha) \imply \beta)$}
\UnaryInfC{$\top _A \leq \Delta _A ^* \forall _{\langle\pi_2,\pi_3\rangle} (\langle\pi_1,\pi_2\rangle^*(\alpha) \imply \beta)$}
\UnaryInfC{$\delta _A \leq \forall _{\langle\pi_2,\pi_3\rangle} (\langle\pi_1,\pi_2\rangle^*(\alpha) \imply \beta)$}
\UnaryInfC{$\langle\pi_2,\pi_3\rangle^* (\delta _A) \leq \langle\pi_1,\pi_2\rangle^*(\alpha) \imply \beta$}
\UnaryInfC{$\langle\pi_2,\pi_3\rangle^* (\delta _A) \wedge \langle\pi_1,\pi_2\rangle^*(\alpha) \leq \beta$}
\UnaryInfC{$\exists _e (\alpha) \leq \beta$}
\DisplayProof
\]
Similarly it can be proved that Frobenius reciprocity is verified (see also \cite{VanOO}, pag 60).
Two important examples of triposes are $\mathbb{H}^{(-)}$, for a complete Heyting algebra $\mathbb{H}$, and $\mathbb{P}(\mathcal{N})^{(-)}$, for a partial combinatory algebra over a set $\mathcal{N}$. In each case $\C$ is $\textbf{Set}$, the category of sets and functions.
There is no need for a tripos to have comprehensions. But this is the case for localic triposes $\mathbb{H}^{(-)}$ and realizability triposes $\mathbb{P}(\mathcal{N})^{(-)}$.
Take a set $X$ and an object $\phi : X \arr \mathbb{H}$: a comprehension of $\phi$ is given by the inclusion $\lfloor \phi \rfloor : \{x \ \epsilon \ X\mid \top \le \phi(x)\} \hookrightarrow X$. The same holds for realizability troposes, for which $\lfloor \phi \rfloor : \{x \ \epsilon \ X\mid \mathcal{N} \subseteq \phi(x)\} \hookrightarrow X$.
These comprehensions can not be full. Take $\lfloor \phi \rfloor : A \hookrightarrow X$ and consider the function $\psi: X \arr \mathbb{H}$ defined by $\psi(x) = \top$ if $x\ \epsilon \ A$ and $\bot$ otherwise. For this function certainly holds $\lfloor \phi \rfloor^* (\psi) = \top$, but it is not the case that $\phi \le \psi$.
\end{exa} 
\begin{exa}\label{TOP}(\textbf{Topologies})
Consider the category $\textbf{TOP}$ of topological spaces and continuos functions. For every topological space $X$, $\mathcal{O}(X)$ is the collection of its open sets, and then it possesses finite meets and arbitrary joins. Take the functor $\mathcal{O}
:\textbf{TOP}^{op} \arr \textbf{ISL}$ determined by the following assignment
$$
\xymatrix@R=1.5ex@C=1.8ex{
(X, \mathcal{O}(X))\ar[dd]_-{f}&&\mathcal{O}(X)\\
&\mapsto&\\
(Y,\mathcal{O}(Y))&&\mathcal{O}(Y)\ar[uu]_-{f^{-1}}
}
$$
Even though each fiber is an Heyting algebra, and therefore it has pseudo relative complements (see \cite{MM}, page 51), $\mathcal{O}$ is not implicational as a doctrine: given a generic continuos function $f$, we have that pseudo relative complements need not commute with reindexing (see \cite{SS}, page 39). $\mathcal{O}$ is existential, since every projection functor has a left adjoint (see \cite{MM}, page 58) satisfying Beck-Chevalley condition and Frobenius reciprocity (recall that projections are open functions). $\mathcal{O}$ has full comprehensions. Given a set $X$, for any open set $S$ in $O(X)$, define its comprehension to be the inclusion function $\lfloor S \rfloor : (S, O_{S}(X)) \hookrightarrow (X, O(X))$, where $O_{S}(X)$ is the topology induced by $S$. These comprehensions are also full. Suppose $Q$ in $O(X)$ such that $\lfloor S \rfloor^{-1}(Q) = S$, this means $\{x\ \epsilon \ S \mid x\ \epsilon \ Q\} = S \cap Q = S$, so $S \subseteq Q$.\\
The doctrine has weak power objects. We call $\Sigma$ the Sierpinski space consisting of two points $0$ and $1$ and a third non trivial open set $\{1\}$. If a topological space $T$ is locally compact, then there exists in \textbf{TOP}  the function space $\Sigma^{T}$ (see \cite{TopHyland} and \cite{topo}). $\Sigma$ extends the subobjects classifier from \textbf{Set} to \textbf{TOP} in the sense that for every $\phi$, open set of $T$, the characteristic function of the inclusion $\lfloor \phi \rfloor$ is the unique arrow making the following a pullback
$$
\xymatrix{
X\ar[d]_-{\lfloor \phi \rfloor}\ar[r]^-{!}&1\ar[d]^-{\top}\\
T\ar[r]_-{\chi_{\phi}}&\Sigma
}
$$
for which it holds that $\chi_{\phi}^{-1}(\{1\}) = \phi$. Now for every topological space $A$ consider any construction that produces a larger locally compact space $\tilde{A}$ such that the inclusion morphism $i_A: A\hookrightarrow \tilde{A}$ is continuos and open, e.g. Alexandroff compactifications, see \cite{topo}; the following lemma holds:
%\item[a)] if $f\times id_B:\tilde{A}\times B \arr Y\times B$ is continuous, then $f_{|_A}\times id_B : A\times B \arr Y\times B$ is continuous
if $f:A\times B \arr \Sigma$ is continuous, then the extension $\tilde{f}: \tilde{A}\times B \arr \Sigma$ is continuous, where $\tilde{f}(a,b) = f(a,b)$ if $a\ \epsilon\ A$, then $\tilde{f}(a,b) = 0$.\\
To prove the lemma it suffices to note that there are no open sets in $\Sigma$ containing the point $0$ other than the top element, then $\tilde{f}^{-1}(\{1\}) = f^{-1}(\{1\})$ and the inclusion function is open. Note that $f=\tilde{f}\circ (i_A\times id)$. Now consider the diagram
$$
\xymatrix{
A\times \Sigma^{\tilde{A}}\ar@{^{(}->}[rr]^-{i_A \times id} &&\tilde{A}\times \Sigma^{\tilde{A}}\ar[r]^-{ev_{\tilde{A}}}&\Sigma\\
A\times B\ar[u]^-{id_{A}\times \overline{\tilde{\chi}_{\phi}}}\ar@{^{(}->}[rr]_-{i_A \times id} && \tilde{A}\times B\ar[u]^-{id_{\tilde{A}}\times \overline{\tilde{\chi}_{\phi}}} \ar[ru]_-{\tilde{\chi}_{\phi}} &
}
$$
define $\in_A := (ev_{\tilde{A}} \circ (i_A \times id))^{-1}(\{1\})$ and for every open set $\phi$ in $A\times B$ define $\{\phi\} := \overline{\tilde{\chi}_{\phi}}$ the exponential transpose of the extension of $\chi_{\phi}$. $(id_A \times \overline{\tilde{\chi}_{\phi}})^{-1}(\in_A) = \chi_{\phi}^{-1}(\{1\}) = \phi$.\\
\\The doctrine fails to be elementary. Given a topological space $X$, we have that $\delta_X$ should be the smallest open set $U$ of $X \times X$ such that $X \subseteq \Delta^{-1}(U)$. In other words $$\delta_X = (\bigcap_{X\subseteq \Delta_X^{-1}(U)}U)^{o}$$ if $X$ is the interval $[0,1]$ with the euclidean topology, then $\delta_X$ would be empty.
\end{exa}
%\begin{exa}\label{sigma}(\textbf{$\sigma$-Algebras})
%Consider the category \textbf{$\sigma$-Alg} of $\sigma$-algebras and measurable functions and take the functor $\textbf{$\sigma$-Alg}^{op} \arr \textbf{ISL}$ determined by the following asignment $(A, \mathcal{F}(A)) \mapsto \mathcal{F}(A)$. This is another non elementary doctrine analogous to the one in example \ref{TOP}.\\
%\end{exa}%%%%%%%%%%%%%%
%%%%%%%%%%%%%%
%%%%%%%%%%%%%%
%%%%%%%%%%%%%%
\section{Quotients and descents}\label{sec2}
 Recall a construction presented in \cite{RM2,RM}, which is based on the notion of equivalence relation in a doctrine.
\begin{deff}\label{equiv} Given a primary doctrine $P:\C^{op}\arr \textbf{ISL}$ and an object $A$ of $\C$, an element $\rho$ in $P(A\times A)$ is said to be an \textbf{equivalence relation} on $A$ if\\
\\
reflexivity: $\top_A \le \Delta_A^* (\rho)$\\
\\
symmetry: $\rho \le \langle \pi_1, \pi_2 \rangle^*(\rho)$\\
\\
transitivity: $\langle \pi_1, \pi_2 \rangle^*(\rho)\wedge \langle \pi_2, \pi_3 \rangle^*(\rho) \le \langle \pi_1, \pi_3 \rangle^*(\rho)$
\end{deff}
Note that if the doctrine $P$ is also elementary, then $\delta_A$ is an equivalence relation on $A$ for every object $A$ in $\C$.\\
\\
In \cite{RM2,RM} the authors consider a certain category $\mathcal{Q}_P$, when $P:\C^{op}\arr \textbf{ISL}$ is elementary. In the category $\Q_P$\\
\\
\textbf{objects} are pairs $(A,\rho)$ such that $\rho$ is an equivalence relation on $A$\\
\\
\textbf{morphisms} $f:(A,\rho)\arr(B,\sigma)$ are %equivalence classes of 
arrows $f:A\arr B$ in $\C$ such that $\rho \le (f\times f)^*\sigma$\\
\\
and composition is given as in $\C$.\\
\\
A first remark is that the construction gives a category in the more general case of $P$ primary. The category $\mathcal{Q}_P$ has binary products: given $(A,\rho)$ and $(B,\sigma)$ in $\mathcal{Q}_P$, $(A,\rho)\times(B,\sigma)\colon =(A\times B, \rho\boxtimes \sigma)$, where $\rho \boxtimes \sigma$ is $\langle\pi_1,\pi_3\rangle^*\rho \wedge\ \langle\pi_2,\pi_4\rangle^*\sigma$. Moreover
if $\C$ has a terminal object, $\mathcal{Q}_P$ has a terminal object.\\
\\
%%%%%%%%%%%%%
%There is an obvious forgetful functor $\mathbb{U}:\Q_P\arr \C$, determined by the following assigment 
%$$\xymatrix@R=1.5ex@C=1.8ex{
%(A, \rho)\ar[dd]_-{f}&&A\ar[dd]^-{f}\\
%&\mapsto&\\
%(B, \sigma)&&B
%}
%$$
%\begin{lem}\label{left} Given an elementary doctrine $P:\C^{op} \arr \textbf{ISL}$, the functor $\mathbb{U}$ has a left adjoint $\nabla$.
%\end{lem}
%\begin{proof} The functor $\nabla$ is defined as
%$$
%\xymatrix@R=1.5ex@C=1.8ex{
%A\ar[dd]_-{f}&&(A, \delta_A)\ar[dd]^-{f}\\
%&\mapsto&\\
%B&&(B, \delta_B)
%}
%$$
%This clearly determine a functor from $\C$ to $\Q_P$ since, for every morphism $f$, $\delta_A \le (f\times f)^* \delta_B$. For every object $(B,\sigma)$ in $\Q_P$, the map $\varepsilon_B \colon = id_B : (B,\delta_B)\arr (B,\sigma)$ is the $B$-component of a natural transformation. This is the counite of the adjunction, since for every object $A$ in $\C$ and every arrow $f:(A,\delta_A)\arr (B,\sigma)$ in $\Q_P$ the diagram commutes
%$$
%\xymatrix{
%(B,\delta_B)\ar[r]^{id_B}&(B,\sigma)\\
%(A,\delta_A)\ar[u]^{f}\ar[ru]_{f}&
%}
%$$
%and $f$ is the unique such arrow.
%\end{proof}
%%%%%%%%%%%%
There is an obvious forgetful functor $\mathbb{U}:\Q_P\arr \C$, and a functor $\nabla: \C \arr \Q_P$, determined by the following assignments
$$\xymatrix@R=1.5ex@C=1.8ex{
&(A, \rho)\ar[dd]_-{f}&&A\ar[dd]^-{f}\\
(\mathbb{U})&&\mapsto&\\
&(B, \sigma)&&B
}
\qquad\ \quad \ \qquad
\xymatrix@R=1.5ex@C=1.8ex{
&A\ar[dd]_-{f}&&(A, \delta_A)\ar[dd]^-{f}\\
(\nabla)&&\mapsto&\\
&B&&(B, \delta_B)
}
$$
$\nabla$ is clearly a functor since, for every morphism $f$ in $\C$, $\delta_A \le (f\times f)^* \delta_B$. 
\begin{lem}\label{left} Given an elementary doctrine $P:\C^{op} \arr \textbf{ISL}$, the functor $\nabla$ is left adjoint to $\mathbb{U}$.
\end{lem}
\begin{proof}
For every object $(B,\sigma)$ in $\Q_P$, the map $\varepsilon_B \colon = id_B : (B,\delta_B)\arr (B,\sigma)$ is the $B$-component of a natural transformation. This is the counite of the adjunction, since for every object $A$ in $\C$ and every arrow $f:(A,\delta_A)\arr (B,\sigma)$ in $\Q_P$ the diagram commutes
$$
\xymatrix{
(B,\delta_B)\ar[r]^{id_B}&(B,\sigma)\\
(A,\delta_A)\ar[u]^{f}\ar[ru]_{f}&
}
$$
and $f$ is the unique such arrow.
\end{proof}
%%%%%%%%%%%%
\begin{deff} Given a primary doctrine $P:\C^{op}\arr \textbf{ISL}$ and an equivalence relation $\rho$ on an object $A$ of $\C$, the poset of descent data $\mathcal{D}es_{\rho}$ is the sub-order of $P(A)$ made by those $\alpha$ such that $$\pi_1^*(\alpha)\ \wedge\ \rho\ \le\ \pi_2^*(\alpha)$$
\end{deff} 
The order $\mathcal{D}es_{\rho}$ is closed under meets and it has trivially $\top_A$, then $\mathcal{D}es_{\rho}$ is an inf-semilattice.\\
The following proposition generalizes to primary doctrines a similar result given for elementary doctrine in \cite{RM2, RM}.
\begin{prop}\label{lefto} Given a primary doctrine $P:\C^{op}\arr \textbf{ISL}$, the assignment
$$
\xymatrix@R=1.5ex@C=1.8ex{
(A,\rho)\ar[dd]_{f}&&\mathcal{D}es_{\rho}\\
&\mapsto&\\
(B,\sigma)&&\mathcal{D}es_{\sigma}\ar[uu]_{f^*}
}
$$
determines a primary doctrine $P_{\mathcal{D}}:\mathcal{Q}_P^{op}\arr \textbf{ISL}$.
\end{prop}
\begin{proof}
It suffices to note that, for every $\beta$ in $\D es_{\sigma}$, $f^* \beta$ is in $\D es_{\rho}$, that can be proved by taking the descent condition on $\beta$, applying to both sides $(f\times f)^*$ and use the fact that $\rho\ \le\ (f\times f)^*\sigma$.
\end{proof}
%%%%%%%%%%%%%%
%%%%%%%%%%%%%%
%%%%%%%%%%%%%%
%%%%%%%%%%%%%%
\section{A co-free construction}\label{sec3}

There is an obvious forgetful functor $\mathcal{U}: \textbf{ED} \arr \textbf{PD}$, which maps every elementary doctrine to itself. We shall show that the construction in \ref{sec2} extends to a 2-right adjoint to it.\\
\\
The following lemma is a strengthening of a similar result in \cite{RM2}.
\begin{lem}\label{elem} Given a primary doctrine $P:\C^{op}\arr \textbf{ISL}$, the doctrine $P_{\mathcal{D}}:\mathcal{Q}_P^{op}\arr \textbf{ISL}$ built as in \ref{lefto} is elementary.
\end{lem}
\begin{proof}
Consider $(A, \rho)$ in $\Q_P$. Note that $\rho$ is an element of $\mathcal{D}es_{\rho \boxtimes\rho}$, since
$$\pi_1^*\rho \wedge (\rho\boxtimes\rho) = \langle \pi_1,\pi_2\rangle^*\rho \wedge\ \langle\pi_1,\pi_3\rangle^*\rho\wedge\ \langle\pi_2,\pi_4\rangle^*\rho$$
and by transitivity of $\rho$
$$\pi_1^*\rho \wedge \rho\boxtimes\rho \ \le \langle\pi_3,\pi_4\rangle^*\rho \ = \ \pi_2^*\rho$$
Let $\delta_{(A,\rho)}$ be $\rho$ and define $\exists_{\Delta_A} \alpha \colon = \pi_1^*\alpha\ \wedge \rho$. We want to prove that, for every $\alpha$ in $\D es_{\rho}$ and $\beta$ in $\D es_{\rho\boxtimes\rho}$, $\exists_{\Delta_A} \alpha \le \beta$ if and only if $\alpha \le \Delta_A^*\beta$. Suppose $\exists_{\Delta_A} \alpha \le \beta$, which means $\pi_1^*(\alpha)\ \wedge \rho\ \le \beta$, and apply $\Delta_A^*$ to both sides, to obtain $\alpha \wedge \Delta_A^*\rho\ \le \Delta_A^*\beta$. So $\alpha \le \Delta_A^*\beta$, by reflexivity of $\rho$. Assume now $\alpha \le \Delta_A^*\beta$, the descent condition for $\beta$ gives:
$$\langle\pi_1,\pi_2\rangle^*\beta\wedge\langle\pi_1,\pi_3\rangle^*\rho\wedge\ \langle\pi_2,\pi_4\rangle^*\rho\ \le\ \langle\pi_3,\pi_4\rangle^*\beta$$
By reindexing along $(\Delta_A\times id_A\times id_A)^*$ and $(\Delta_A\times id_A)^*$ one obtains $$\pi_1^* \Delta_A^* \beta \ \wedge \rho\ \le \beta$$ by reflexivity of $\rho$
\[
\AxiomC{$\alpha \le \Delta_A^* \beta$}
\UnaryInfC{$\pi_1^*\alpha \le \pi_1^*\Delta_A^* \beta$}
\UnaryInfC{$\pi_1^*\alpha \ \wedge \ \rho \ \le \  \pi_1^*\Delta_A^* \beta \ \wedge \ \rho$}
\AxiomC{$\pi_1^* \Delta_A^* \beta \ \wedge \rho\ \le \beta$}
\BinaryInfC{$\pi_1^*\alpha \ \wedge \ \rho \ \le \ \beta$}
\UnaryInfC{$\exists_{\Delta_A}\alpha \ \le \ \beta$}
\DisplayProof
\]
%%%%%%%%
To verify the conditions ii) of \ref{ele}, consider an object $(X, \tau)$ and let $e\colon = id_X \times \Delta_A$ be a morphism in $\mathcal{Q}_P$. The proof that if  $\exists_e(\alpha) \le \beta$, then $\alpha \le e^*(\beta)$ is similar to that in example \ref{Tripos} (where $\rho$ is $\delta_A$). The proof of the converse, is essentially as before where: $$\langle\pi_1,\pi_2,\pi_3,\rangle^*\beta \ \wedge\ \langle\pi_1,\pi_4\rangle^*\tau \ \wedge \ \langle \pi_2,\pi_5\rangle^*\rho \ \wedge \ \langle\pi_3,\pi_6\rangle^*\rho\ \le \ \langle\pi_4,\pi_5,\pi_6\rangle^* \beta$$ and reindexing along the following composition
$$
\xymatrix{
X\times A\times A\ar[d]^-{id_X\times \Delta_A \times id_A}&&X\times A\times A\times X\times A\times A\\
X\times A\times A \times A\ar@/_1.5pc/[dr]^-{\ \ \ \ \ \Delta_X \times id_A\times id_A \times id_A}&&X\times A\times X\times A \times A\ar[u]^-{id_X \times \Delta_A \times id_X \times id_A \times id_A}\\
&X\times X\times A\times A \times A\ar@/_1.5pc/[ru]^-{id_X \times tw \times id_A \times id_A\ \ \ \ \ \ \ }&
}
$$
%to obtain the inequality $$\langle\pi_1,\pi_2\rangle^* (id_X \times \Delta_A)^* \beta \ \wedge \ \langle \pi_2, \pi_3 \rangle^*\rho\ \le \ \beta$$
%which we use in the following tree
%\[
%\AxiomC{$\alpha \le (id_X \times \Delta_A)^* \beta$}
%\UnaryInfC{$\langle\pi_1,\pi_2\rangle^*\alpha \ \le \  \langle\pi_1,\pi_2\rangle^*(id_X\times \Delta_A)^* \beta$}
%\UnaryInfC{$\langle\pi_1,\pi_2\rangle^*\alpha \ \wedge \ \langle\pi_2,\pi_3\rangle^*\rho\ \le \  \langle\pi_1,\pi_2\rangle^*(id_X\times \Delta_A)^* \beta \ \wedge \ \langle\pi_2,\pi_3\rangle^*\rho$}
%\UnaryInfC{$\langle\pi_1,\pi_2\rangle^*\alpha \ \wedge \ \langle\pi_2,\pi_3\rangle^*\rho \ \le \ \beta$}
%\UnaryInfC{$\exists_{e}\alpha \ \le \ \beta$}
%\DisplayProof
%\]
\end{proof}

%Recall from section \ref{sec3} that, given a primary doctrine $P:\C^{op}\arr \textbf{ISL}$, the doctrine $P_\mathcal{D}$, which maps every object of $\Q_P$ in the inf-semilattice of descent data, is elementary (proposition \ref{elem}).\\ 
Given a 1-morphism in $\PD$, $(F,f) : P \arr R$, consider the functor $F_\mathcal{D}$ defined by the following assignment
$$
\xymatrix@R=1.5ex@C=1.8ex{
(A,\rho)\ar[dd]_{q}&&(FA, <\pi_1,\pi_2>^*f_{A \times A} (\rho))\ar[dd]^{Fq}\\ 
&\mapsto&\\
(B,\sigma)&&(FB, <\pi_1,\pi_2>^*f_{B \times B} (\sigma))
}
$$
and the $\Q_P$-indexed family of arrow $f_\mathcal{D}$ whose $(A, \rho)$-component is the restriction of $f_A: P(A)\arr R(FA)$ to $\mathcal{D}es_{\rho}$
\begin{lem}\label{elemor}
Given a 1-morphism in \textbf{PD}, $(F,f) : P \arr R$ the pair $(F_\mathcal{D},f_\mathcal{D}):P_\mathcal{D} \arr R_\mathcal{D}$ determines a 1-morphism in \textbf{ED}.
\end{lem}
\begin{proof}
 First note that $<\pi_1,\pi_2>^*f_{A \times A} (\rho)$ is an equivalence relation since $\rho$ is and $f$ is natural. $Fq$ is a morphism in $\Q_P$, since $<\pi_1,\pi_2>^*f_{A \times A} (\rho)\le(Fq\times Fq)^*<\pi_1,\pi_2>^*f_{B \times B} (\sigma)= <\pi_1,\pi_2>^*F(q\times q)^*f_{B \times B} (\sigma)= <\pi_1,\pi_2>^*f_{A \times A} (q\times q^*\sigma)$, for naturality of $f$. It is left to show that the images of the restriction is $\mathcal{D}es_{<\pi_1,\pi_2>^*f_{A\times A}(\rho)}$, but this is true since, for $\alpha$ in $\D es_{\rho}$, $\pi_1^* \alpha\ \wedge \rho\ \le \pi_2^* \alpha$, then apply $f_{A\times A}$ to both sides and, recalling that $f_{A\times A}\circ \pi_1^* = \pi_1^* \circ f_A$ for naturality of $f$,  one has $\pi_1^*f_A^* \alpha\ \wedge\ f_{A\times A}(\rho)\ \le\ \pi_2^*f_A^* \alpha$. Now it suffices to reindex both sides along $<\pi_1,\pi_2>$. The last step is to show that $f_\D$ preserves the elementary structure, i.e. $f_{\D(A,\rho)\times(A,\rho)}(\delta_{(A,\rho)})=<F_\D \pi_1,F_\D \pi_2>^*(\delta_{F_\D(A,\rho)})$, which reduces to the following equality $f_{A\times A}(\rho)=<F \pi_1,F \pi_2>^*(<\pi_1,\pi_2>^*f_{A \times A} (\rho))$, where $<\pi_1,\pi_2>\circ <F \pi_1,F \pi_2> = id_{F(A\times A)}$.
%Lemma \ref{elem} ensures that $<\pi_1,\pi_2>^*f_{A\times A} (\rho)$ is $\delta_{FA}$.
\end{proof}
Consider the functor $(-)_\mathcal{D}:\PD\arr \ED$ 
$$
\xymatrix@R=1.5ex@C=1.8ex{
P\ar[dd]_-{(F,f)}&&P_\D \ar[dd]^-{(F_\D,f_\D)}\\
&\mapsto&\\
R&&R_\D
}
$$
For every doctrine $P: \C\arr\textbf{ISL}$ in \textbf{PD} there is a 1-morphism $\varepsilon_P$ from $P_\D$ to $P$ given by the pair $(\mathbb{U}, i)$, where $\mathbb{U}:\Q_P \arr \mathbb{C}$ is the forgetful functor defined before \ref{left}, while the $A$-component of $i$ is the inclusion functor $\D es_{\rho} \hookrightarrow P(A)$.
\begin{prop} The natural transformation $\varepsilon$ is the counit of an adjunction $\mathcal{U}\dashv (-)_\mathcal{D}$.
\end{prop}
\begin{proof}
Note that $\mathcal{U}(P ) = P$; given an elementary doctrine $P:\C^{op}\arr \textbf{ISL}$, a morphism $(F,f): \mathcal{U}(P ) \arr R$ in $\PD$, consider the arrow $(\overline{F}, \overline{f}): P \arr R_\D$ in $\ED$, determined by the following composition
$$
\xymatrix{
\C^{op}\ar@/^1pc/[rrrd]^-{P}="a"\ar[d]_-{\nabla}&&&\\
\Q_P^{op}\ar[rrr]_-{P_\D}="b"\ar[d]_-{F_\D}&&&\textbf{ISL}\\
\Q_R^{op}\ar@/_1pc/[rrru]_-{R_\D}="c"&&&
\ar"a";"b"^{id_{PA}}\ar"b";"c"^-{f_\D}}
$$
then $\overline{F} := F_\D \circ \nabla$ and $\overline{f} := f_\D \circ id_{PA}$. Where the natural transformation $P \arr P_\D \circ \nabla$ is the identity from the fact that $\D es_{\delta_A} = P(A)$. What is left to prove is that $(\overline{F},\overline{f})$ is the unique arrow that makes the following diagram commutes
$$
\xymatrix{
\Q_R^{op}\ar[rd]^-{R_\D}\ar@/^/[rr]^-{(\mathbb{U},i)}&&\mathbb{D}^{op}\ar[dl]^-{R}\\
&\textbf{ISL}&\\
&\C^{op}\ar[u]^-{P}\ar@/_1pc/[ruu]_-{(F,f)}\ar@/^1pc/[uul]^{(\overline{F},\overline{f})}&
}
$$
Commutativity: recall that,  for an object $A$ in $\C$,  $\mathbb{U}(\overline{F})(A)$ is $\mathbb{U}(F_\D(\nabla (A)))$, then follow the assignments below
$$A\mapsto (A,\delta_A)\mapsto (FA,  \delta_{FA}) \mapsto FA$$ moreover $(i \circ \overline{f})_A$ is $i_A \circ f_{\D A} \circ id_{PA}$, then take $\alpha$ in $P(A)$ and follow the assignments
$$\alpha \mapsto f_A (\alpha) \mapsto i(f_A(\alpha))=f_A(\alpha)$$
Uniqueness is given by the fact that $(\mathbb{U}, i)$ is mono, since $\mathbb{U}$ is the identity on objects and morphism and $i$ is an inclusion functor.
\end{proof}
%%%%%%%%%%%%%%
%%%%%%%%%%%%%%
%%%%%%%%%%%%%%
%%%%%%%%%%%%%%
\section{Applications}\label{sec4}
The co-free construction presented in the previous section preserves all the first order predicate structures which are in $P$ in the sense of the following
\begin{prop}\label{des} Given a primary doctrine $P:\C^{op} \arr \textbf{ISL}$ and the elementary doctrine $P_\D:\Q_P^{op}\arr\textbf{ISL}$
\begin{itemize}
\item[(i)] if $P$ has finite distributive joins, so has $P_\D$ and $\varepsilon_P : P_\D \arr P$ preserves them
\item[(ii)] if $P$ is implicational, so is $P_\D$ and $\varepsilon_P$ preserves this
\item[(iii)] if $P$ existential, so is $P_\D$ and $\varepsilon_P$ preserves this
\item[(iv)] if $P$ universal, so is $P_\D$ and $\varepsilon_P$ preserves this
\item[(v)] if $P$ is has (full) comprehensions, so has $P_\D$ and $\varepsilon_P$ preserves them
\end{itemize}
\end{prop}
\begin{proof} (i) Given $\alpha$ and $\beta$ in $P_\D(A,\rho)$, the join $\alpha \lor \beta$ in $P(A)$ is in $\mathcal{D}es_\rho$ by distributivity. (ii) Like before, given $\alpha$ and $\beta$ in $P_\D(A,\rho)$, take $\alpha \imply \beta$ in $P(A)$. To see this is in $P_\D(A,\rho)$, recall that, since $\rho$ is symmetric, the descent condition can be written as $\pi_2^* \alpha \wedge \rho = \pi_1^*\alpha \wedge \rho$. One has that $\pi_1^*(\alpha \imply \beta) \wedge \rho \le \pi_2^*(\alpha \imply \beta)$ if and only if $\pi_1^*(\alpha \imply \beta) \wedge \rho \wedge \pi_2^*\alpha \le \pi_2^*\beta$ if and only if $\pi_1^*(\alpha \imply \beta) \wedge \pi_1^*\alpha \wedge \rho \le \pi_2^*\beta$. (iii) For $\alpha$ in $P_\D(A\times B, \rho \boxtimes \sigma)$, we have $\pi_1^*\exists_{\pi1} (\alpha) \wedge \rho = \exists_{\langle\pi_1,\pi_2\rangle}<\pi_1,\pi_3>^*(\alpha) \wedge \rho$ by Beck-Chevalley. By Frobenius Reciprocity that is equal to $\exists_{\langle\pi_1,\pi_2\rangle}(\langle\pi_1,\pi_3\rangle^*\alpha \wedge \langle\pi_1,\pi_2\rangle^*\rho) \le \exists_{\langle\pi_1,\pi_2\rangle}\langle\pi_2,\pi_3\rangle^*\alpha = \pi_2^*\exists_{\pi1}\alpha$; (iv) we have that $\pi_1^*\forall_{\pi1} (\alpha) \wedge \rho \le \pi_2^*\forall_{\pi1} (\alpha)$ if and only if $\forall_{<\pi_1,\pi_2>} <\pi_1,\pi_3>^*(\alpha) \wedge \rho \le \forall_{<\pi_1,\pi_2>} <\pi_2,\pi_3>^*(\alpha)$. Since $<\pi_1, \pi_2>^* \dashv \forall_{<\pi_1,\pi_2>}  $ the inequality holds if and only if $<\pi_1,\pi_3>^* \alpha \wedge <\pi_1,\pi_2>^*\rho \le <\pi_2,\pi_3>^*\alpha$ which is the descent condition for $\alpha$. (v) Take an element $\alpha$ in $P_\D(A,\rho)$, this is also in $P(A)$, and consider its comprehension $\lfloor\alpha \rfloor: X\arr A$, this produces a comprehension morphism $(X, (\lfloor\alpha \rfloor \times \lfloor\alpha \rfloor)^*\rho)\arr(A,\rho)$ in $\Q_P$. Fullness directly derives from that in $P$. In each case (i)-(iv) we shall show that $P_\D(A, \rho) = \mathcal{D}es_\rho \subseteq P(A)$ is closed under the relevant constructions, thus obtaining immediately preservation by $\varepsilon_P$.
\end{proof}
%%%%%%%%%%%%%%%
In Example \ref{TOP} we presented a doctrine that fails to be implicational since, even though every fiber has pseudo relative complements, they do not distribute under reindexing. Moreover the doctrine is not universal: it has right adjoints along all the projections, but these do not satisfied Beck-Chevalley conditions. The next two propositions show that these two properties are gained with the co-free construction. The first is from \cite{LawDiag} and the second is standard.
\begin{prop}\label{asaf}
If $P:\C^{op} \arr \textbf{ISL}$ is an elementary existential doctrine and every fiber has pseudo relative complements, then $P$ is implicational. 
\end{prop} 
\begin{proof}
Suppose $f:A\arr B$ is a morphism in $\C$, by \ref{GQ} there exists $\exists_f :P(A)\arr P(B)$ statisfying Frobenius Reciprocity.
\[
\AxiomC{$(f^*\alpha \imply f^*\beta) \wedge f^*\alpha \le f^*\beta$}
\UnaryInfC{$\exists_f((f^*\alpha \imply f^*\beta) \wedge f^*\alpha) \le \beta$}
\UnaryInfC{$\exists_f(f^*\alpha \imply f^*\beta) \wedge \alpha \le \beta$}
\UnaryInfC{$\exists_f(f^*\alpha \imply f^*\beta) \le \alpha \imply \beta$}
\UnaryInfC{$f^*\alpha \imply f^*\beta \le f^*(\alpha \imply \beta)$}
\DisplayProof
\]
To prove that $f^*(\alpha \imply \beta) \le f^*\alpha \imply f^*\beta $ it suffices to use the distributivity of reindexing functors on meets.
\end{proof}
\begin{prop}\label{avidian}
If $P:\C^{op} \arr \textbf{ISL}$ is an existential elementary doctrine with right adjoints $\forall_{\pi}$ along every projection $\pi$, then $P$ is universal. 
\end{prop}
%\begin{proof}
%We need show that for every projection $\pi_B:X\times B \arr B$ in $\C$, the right adjoint $\forall_{\pi_B}$ satisfies Beck Chevalley. Take an element $\alpha$ in $P(X\times B)$ and a morphism $f:A\arr B$ in the base category and recall that $f^*$ has a left adjoint $\exists_f$ satisfying Beck-Chevalley as in $\ref{GQ}$.
%\[
%\AxiomC{$\pi_A^*\forall_{\pi_A} (id_X \times f)^*\alpha \le (id_X\times f)^*\alpha$}
%\UnaryInfC{$\exists_{(id_X\times f)} \pi_A^*\forall_{\pi_A} (id_X \times f)^*\alpha \le \alpha$}
%\UnaryInfC{$\pi_B^*\exists_f \forall_{\pi_A} (id_X \times f)^*\alpha \le \alpha$}
%\UnaryInfC{$\exists_f \forall_{\pi_A} (id_X \times f)^*\alpha \le \forall_{\pi_B} \alpha$}
%\UnaryInfC{$\forall_{\pi_A} (id_X \times f)^*\alpha \le f^*\forall_{\pi_B} \alpha$}
%\DisplayProof
%\]
%The other inequality $f^*\forall_{\pi_B} \alpha\le\forall_{\pi_A} (id_X \times f)^*\alpha$ is canonical.
%\end{proof}
As a corollary of \ref{des} and \ref{asaf}, we have that if $P$ is an existential doctrine and every fiber has pseudo relative complements, then $P_\D$ is implicational. And, as a corollary of \ref{des} and \ref{avidian}, if $P$ is an existential doctrine with right adjoints along every projections, then $P_\D$ is universal. In particular the doctrine $\mathcal{O}_\D$ is implicational and universal.\\
%%%%%%%%%%%%%%%
\\Power objects are not preserved, but it holds that
\begin{prop}\label{h1} If $P:\C^{op}\arr \textbf{ISL}$ is universal and implicational with weak power objects, then $P_\D$ has weak power objects.
\end{prop}
\begin{proof} A weak power object of $(A, \rho)$ in $\mathcal{Q}_P$ is $$(\pi A, \forall_{<\pi_2,\pi_3>}(<\pi_1,\pi_2>^* \in_A \sse <\pi_1,\pi_3>^* \in_A ))$$
where the membership predicate $\in_{(A,\rho)}$ is $$\in_A \wedge\ \forall_{<\pi_1, \pi_3>}(<\pi_1, \pi_2>^*\rho \imply <\pi_2, \pi_3>^* \in_A)$$.
\end{proof}
It is worth to remark that power objects as defined in \ref{h1} are still weak, but they gain the property that, in $\Q_P$ if two morphisms $\{\phi\}$ and $\{\phi\}'$ classify the same element $\phi$ in the fiber over $(A,\rho) \times (B,\sigma)$, then it holds that $\top_B \le <\{\phi\},\{\phi\}'>^*\delta_{\pi(A,\rho)}$. This lead to introduce internal extentionality. We said that for an object $A$ in the base category of an elementary doctrine, $\delta_A$ provides a notion of internal equality for terms of type $A$. Certainly external equality implies internal, in the sense that given $t_1,t_2:X\arr A$, if it holds that $t_1=t_2$ (i.e they are the same morphism in $\C$) then $\top_X \le <t_1,t_2>^*\delta_A$. The converse can be forced considering the category $[\C]$, whose objects are the same as in $\C$ and the morphism are equivalence classes of morphism of $\C$ with respect to the relation: $[t_1]=[t_2]$ if and only if $\top_X\ \le <t_1,t_2>^*\delta_A$. This construction is given directly in \cite{RM2,RM}, and named extentional collapse of $\C$.\\
\\If $P$ is an elementary doctrine with power objects in which the base category $\C$ has a terminal object $1$, then for every object $A$ in $\C$, every element $\phi$ determines (at least) a term of type $\pi1$, i.e. $\{\phi\}: A\arr \pi1$ (see \ref{p1}). Hence we can use this correspondence to define a notion of internal equality for formulas $$\phi \leftrightarrow \psi\ :=\langle\{\phi\},\{\psi\}\rangle^*\delta_{\pi1}$$ which depends on a choice of the morphisms $\{\phi\}$ and $\{\psi\}$ and satisfies the following rule
\[
\AxiomC{$\gamma \le \phi \leftrightarrow \psi$}
\UnaryInfC{$\gamma \wedge \psi \le \phi$\ \quad \ $\gamma \wedge \phi \le \psi$}
\DisplayProof
\]
by the fact that we have $\pi_1^*x\ \wedge\ \delta_{\pi1} \le\ \pi_2^*x$, for every $x$ in $P(\pi1)$, then reindex both sides along $\langle\{\phi\},\{\psi\}\rangle$ with $x = \epsilon_1$ to have $\phi \wedge (\phi \leftrightarrow \psi) \le \psi$, which we use in the following tree
\[
\AxiomC{$\gamma \le \phi \leftrightarrow \psi$}
\UnaryInfC{$\gamma \wedge \phi \le (\phi \leftrightarrow \psi) \wedge \phi$}
\UnaryInfC{$\gamma \wedge \phi \le \psi$}
\DisplayProof
\]
The converse of the previous rule does not holds in general. This motivates the following
\begin{deff} Given a primary doctrine $P:\C^{op}\arr \textbf{ISL}$ in which the base category has a terminal object $1$, a weak power object $\pi1$ and an elementary structure $\delta_{\pi1}$ in the poset over $\pi1\times \pi1$, we say that $P$ has \textbf{extentional entailment} if, for every object $A$ in $\C$ and every element $\phi$, $\psi$ and $\gamma$ in $P(A)$ the following rule
\[
\AxiomC{$\gamma \wedge \psi \le \phi$}
\AxiomC{$\gamma \wedge \phi \le \psi$}
\doubleLine
\BinaryInfC{$\gamma \le <\{\phi\},\{\psi\}>^*\delta_{\pi1}$}
\DisplayProof
\]
is satisfied
\end{deff}
As an immediate property we have that in an elementary doctrine $P:\C^{op}\arr\textbf{ISL}$ with extentional entailment, for every formulas $\phi$ and $\psi$, it holds that $\phi \leftrightarrow \psi$ if and only if $\top \le <\{\phi\},\{\psi\}>^*\delta_{\pi1}$, which means that every classifying morphism is unique in the exentional collapse of $\C$.
\begin{Remark}\label{imply} There is a connection between extentional entailment and the presence of pseudo relative complements in every fibre of a doctrine. If a doctrine $P$ has extentional entailment, then for every object $A$ in $\C$, $P(A)$ has pseudo relative complements: it suffices to define $\alpha \imply \beta\ := (\alpha \wedge \beta)\leftrightarrow\alpha$, in the spirit of logic of toposes (see \cite{BoiJoy}). If $P$ is elementary, then is also implicational by \ref{asaf}.
The converse need not to be true in the sense that, even if an elementary doctrine has pseudo relative complements over each fiber, we have that $\pi_1^*\epsilon_1\sse\pi_2^*\epsilon_1$, may not be the left adjoint to $\Delta_{\pi1}^*$, as we see, for instance, in example \ref{extri}.
\end{Remark}
\begin{exa} (\textbf{Subobjects}) Let $\C$ be a finitely complete small category. The doctrine $\textbf{Sub}:\C^{op}\arr\textbf{ISL}$ has extentional entailment if and only if $\C$ has a subobjects classifier. Let $\Omega$ be the subobjects classifier of $\C$. Then $\pi1$ is $\Omega$. $\epsilon_1$ is the true arrow. To prove the converse, suppose \textbf{Sub} to have extentional entailment. Define $\Omega$ to be $\pi1$. The true arrow is $\epsilon_1: 1\arr\pi1$. Every mono $\phi:X\arr A$ is classified by $\{\phi\}$, since reindexing is given by pullbacks. $\{\phi\}$ is unique because of extentionality of entailment, which says that if $\phi = f^*\epsilon_1$ for some $f$, then $\top \le <\{\phi\}, f>^*\delta_{\Omega} $, where $\delta_{\Omega}$ is $\Delta_{\Omega}$. Under the same conditions, an immediate corollary is that $\C$ is an elementary topos if and only if \textbf{Sub} has power objects. Suppose \textbf{Sub} to have power objects. For each $A$ in $\C$, $\Omega^{A}$ is $\pi A$ and $ev_A$ is $\{\in_A\}$. For every morphism $f : A\times B\arr \Omega$, the transpose $\overline{f}$ is $\{f^*\epsilon_1\}$. Then $\{\in_A\} \circ (id_A\times\{f^*\epsilon_1\})=f$, since they both classify $f^*\epsilon_1$. The converse is proved in $\cite{Jacobs}$, page 336.
\end{exa}
\begin{exa} (\textbf{Triposes})\label{extri} In general a tripos need not have extentional entailment. Take the localic tripos. $1=\{*\}$. $\delta_{\pi_1}:\mathbb{H}^1\times \mathbb{H}^1 \arr \mathbb{H}$ is given by the following assignment $(f,g)\mapsto \top$ if $f=g$, then $\bot$. But $f(*)\sse g(*)$ is not necesserly $\bot$ if $f \not = g$. Analogously for realizability triposes, see \cite{Jacobs} page 331.
\end{exa}
In \ref{imply} and \ref{extri} we showed that an implicational doctrine need not have an extentional entailment. The following proposition says that this holds once the elementary structure is co-freely added to a doctrine. In other words given a  doctrine $P$, the canonical inequality $\delta_{\pi1} \le \pi_1^*\epsilon_1 \sse \pi_2^*\epsilon_1$ is an equality in $P_\D$. 
\begin{prop}\label{h2} If $P:\C^{op}\arr \textbf{ISL}$ is such that $\C$ has a terminal object $1$ with a weak power object $\pi1$ and pseudo relative complements in $P(\pi1)$, then  $P_\D:\Q_P\arr\textbf{ISL}$ has extentional entailment.
\end{prop}
\begin{proof} Recalling that in $\Q_P$ the terminal object is $(1, \top_{1\times1})$, define $\pi(1,\top_{1\times1})\ :=(\pi1,\pi_1^*\epsilon_1\sse\pi_2^*\epsilon_1)$. $\epsilon_1$ certainly belongs to the category of descent data, since $\pi_1^*\epsilon_1 \wedge (\pi_1^*\epsilon_1\sse \pi_2^*\epsilon_1) \le \pi_2^*\epsilon_1$. To prove that $P_\D$ has extentional entailment it is left to show that $\pi_1^*\epsilon_1\sse\pi_2^*\epsilon_1$ is $\delta_{\pi(1,\top_{1\times1})}$, which is true by proposition \ref{elem}, .
\end{proof}
As a final remark note that, under the hypothesis of proposition \ref{h1}, the elementary doctrine $P_\D$ is a tripos, since it is possible to define finite joins and  existential quantifications on the basis of implicational operations and (higher order) universal quantifications (See \cite{Tripos}, \cite{Jacobs} and  \cite{McLarty}).
 Then a doctrine that differs from a tripos only by the lack of an elementary structure, thanks to propositions \ref{des} and \ref{h1} comes to be a tripos and this tripos has extentional entailment by \ref{h2}.
 On the other hand any tripos $P$, which is known to be an interpretation of higher order many-sorted non-extentional predicate logic (see \cite{Jacobs} or \cite{VanOO}), generetes a new tripos $P_\D$ which interpretes higher order many-sorted predicate logic with extentional entailment.
\bibliography{biblio}
\bibliographystyle{plain}

\end{document}